\documentclass[a4paper,12pt]{amsart}

\usepackage{amssymb,amsbsy,amsmath,amsfonts,amssymb,amscd}
\usepackage{latexsym}
\usepackage{graphics}
\usepackage{color}
\usepackage{comment}
\input xy
\xyoption{all}

\theoremstyle{plain}
\newtheorem{thm}{Theorem}[section]
\newtheorem{theorem}[thm]{Theorem}

\newtheorem{lemma}[thm]{Lemma}
\newtheorem{corollary}[thm]{Corollary}
\newtheorem{proposition}[thm]{Proposition}
\theoremstyle{definition}
\newtheorem{remark}[thm]{Remark}

\newtheorem{defin}[thm]{Definition}

\newtheorem{assumption}[thm]{Assumption}

\newtheorem{example}[thm]{Example}

\numberwithin{equation}{section}

\newcommand{\sB}{{\mathcal B}}
\newcommand{\sC}{{\mathcal C}}

\newcommand{\sE}{{\mathcal E}}
\newcommand{\sF}{{\mathcal F}}

\newcommand{\sL}{{\mathcal L}}

\newcommand{\sP}{{\mathcal P}}
\newcommand{\sQ}{{\mathcal Q}}
\newcommand{\sR}{{\mathcal R}}
\newcommand{\sS}{{\mathcal S}}
\newcommand{\sT}{{\mathcal T}}

\newcommand{\PP}{\ensuremath{\mathbb{P}}}

\newcommand{\CC}{\ensuremath{\mathbb{C}}}

\newcommand{\ZZ}{\ensuremath{\mathbb{Z}}}

\newcommand{\hol}{\ensuremath{\mathcal{O}}}

\newcommand\la{\lambda}
\newcommand\s{\sigma}

\newcommand\ga{\gamma}
\newcommand\de{\delta}
\newcommand\e{\epsilon}

\newcommand{\ra}{\ensuremath{\rightarrow}}

\def\eea{\end{eqnarray*}}
\def\bea{\begin{eqnarray*}}

\newcommand\dual{\mathrel{\raise3pt\hbox{$\underline{\mathrm{\thinspace d
\thinspace}}$}}}
\newcommand\qe{\ifhmode\unskip\nobreak\fi\quad $\Box$}       

\def\BOX{\hfill\lower.5\baselineskip\hbox{$\Box$}}

\newtheorem{theo}{Theorem}[section]

\newtheorem{remarkk}[theo]{Remark}
\newenvironment{rem}{\begin{remarkk}\rm}{\end{remarkk}}

\usepackage{hyperref}
\title [Genus $2$ pencils for $p_g=K^2=1$ and plane cubics. ]{ Genus $2$ pencils on  surfaces with $p_g=K^2=1$,    envelopes, and conics tangent to plane cubic curves.}

\author{Fabrizio Catanese, }
\address {Accademia Nazionale dei Lincei, Palazzo Corsini, via della Lungara 10, 00165 Roma, Italia.}
\email{Fabrizio.Catanese@gmail.com, Fabrizio.Catanese@lincei.it}
\author{Part (II)    coauthored by  Noah Ruhland}
\address {Mathematisches Institut der Universit\"at Bayreuth\\
NW II,  Universit\"atsstr. 30\\
95447 Bayreuth}
\email{noah.ruhland@uni-bayreuth.de}
 
 \date{\today}

\begin{document}

\maketitle

\begin{abstract}
 We consider $(1,1)$-surfaces, namely, 
minimal compact complex surfaces $S$
with $p_g (S) =K_S^2=1$: for these the bicanonical map is a covering of degree $4$ of the plane  $\PP^2$.

And we answer a question posed by Meng Chen, whether they can contain a genus 2 pencil (this is the standard reason
of failure of  birationality of the bicanonical map).

Our main theorem says that those which admit a  genus 2 pencil form an irreducible  subvariety
of codimension $3$ in their moduli space  $\frak M_{[1,1]}$; moreover, the general such surface admits exactly $24$ such pencils, corresponding to $12$ pencils of conics in $\PP^2$.

The real fun is to relate this variety to the geometry of pencils of conics in the plane
everywhere tangent to a cubic curve and a line. We investigate the corresponding variety $\sT$ of triples
and provide explicit equations using the classical theory of envelopes:  among others, equations given
in terms  of  the Weierstrass normal form of the cubic.
\end{abstract}

\bigskip

\addtocontents{toc}{\protect\setcounter{tocdepth}{1}}
\tableofcontents

\section{Part I: standard and non standard case of non-birationality of the bicanonical map}

\section{Introduction}

In the theory of algebraic surfaces $S$ of general type, while the non-birationality of the pluricanonical maps 
$\Phi_m$, for $ m \geq 3$, occurs only for $ m \leq 5$ and for a finite number of families of surfaces
(see \cite{superficie}, \cite{cm}), 
the bicanonical map $\Phi_2$  is  not birational 
in the  so-called {\bf standard case} where the surface $S$ admits a pencil of curves of genus $g=2$.
Because the bicanonical map of a smooth curve $C$ is not birational if and only if $C$ has genus $\leq 2$.

This exception gives rise to an  infinite number of families.

Instead,  the  study of {\bf non-standard cases} for the non-birationality of the bicanonical map, in view of Bombieri's result \cite{cm} and  improvements by Reider \cite{reider} and others starting from \cite{3authors}, boils down to the analysis of a finite number of families of 
 surfaces with low invariants, namely with $p_g (S) \leq 6$ and $K_S^2 \leq 9$,  (we refer to \cite{ciliberto} for a survey concerning non-standard cases
of  non-birationality of $\Phi_2$, and further references therein).

This study turns out to be quite useful in the analysis of pluricanonical maps of threefolds and higher dimensional varieties,
where by induction one reduces to lower dimensional varieties.
This article was indeed motivated by a question raised by Meng Chen. 
He  asked whether, if  $S$ is a $(1,1)$-surface, that is, a complex minimal surface of general type with $p_g(S)=1, K^2_S=1$, then   necessarily $S$  does not admit any   pencil of curves of geometric genus 2. We show here that  such a genus two pencil exists exactly 
for $S$ in a proper subvariety of the moduli space.\footnote{  A referee pointed out that the answer to this weaker question about the existence of such surfaces with a genus two pencil could have been derived by using
some surfaces $W$ constructed in Example 5.3 of \cite{rito}, pages 100-101.} 

For $(1,1)$-surfaces  the bicanonical map $\Phi_2 : S \ra \PP^2$ is a  degree $4$ morphism, hence 
 for the general surface we have a non-standard case of non-birationality, 
and  both the standard and the non-standard case    occur in the same irreducible family.

Our main result  is the precise answer to the above question, summarized by  the following  theorem:

\begin{theorem}\label{main}
(I) Let $\frak M_{[1,1]}$ be the moduli space of  complex minimal surfaces $S$ of general type with $p_g(S)=1, K^2_S=1$, which is irreducible of 
dimension $18$. 

Then   the general surface $S \in \frak M_{[1,1]}$  does not admit a   free genus $2$ pencil,  that is,
a fibration $ \Psi : S \ra C$ with general fibre $F$ of  genus $g=2$.

(II) In particular, the general surface $S \in \frak M_{[1,1]}$  does not admit a   genus $2$ pencil.

(III) The subvariety $ \frak N_2 \subset \frak M_{[1,1]}$ consisting of the surfaces $S$ admitting a free genus $2$ pencil 
is a dense subset of  the  irreducible subvariety  $ \frak N'_2$, of dimension 15,   
  consisting of the surfaces  whose 
  canonical model $X \subset \PP (1, 2,2,3, 3)$ admits
 equations, in the  coordinates 
$ (v, x_2,  x_3, z_3, z_4),$ 
and setting $  x_1 : = v^2, \ x := (x_1,x_2,x_3),$ in the normal form  as in \eqref{specialeqs}
$$ z_3^2 + c_1(x)= 0, z_4^2  + a_2 (x) v z_3 + c_2(x)=0. $$
These equations  exhibit the  bicanonical morphism as an iterated double covering of the plane 
factoring through a degree 2 Del Pezzo surface $Y$, 
$$Y : = \{ w^2 +  x_1 c_1(x) =0\}.$$

Moreover, letting  $ \frak N_2(0)  \subset   \frak N_2$ be the  intersection with the  open set consisting of the surfaces
$S$ with  $K_S$ ample, and defining similarly $ \frak N'_2(0)  \subset   \frak N'_2$,  $ \frak N'_2(0)$ is non empty and 
we have equality $ \frak N_2(0)  =  \frak N'_2 (0).$

(IV) The general surface $S$ in $ \frak N_2$  admits exactly $24$ genus $2$ pencils.

\end{theorem} 

Statement (II) follows from a simple observation by
 Meng Chen,  that,   if $S$ admits  a genus $2$ pencil, then  this pencil must necessarily be a free genus $2$ pencil.

The proof of Theorem \ref{main} is based  on three steps:  first we show that the existence of a genus $2$ pencil implies that
the surface belongs to the irreducible subvariety $\frak N'_2$.

Secondly, we show that, for a general surface in $\frak N'_2$,  the existence of a genus $2$ pencil 
is equivalent  to a statement about the existence of a 1-parameter family (pencil)  of conics $Q_\la$ in the plane which satisfy some conditions, the main one being  (*):  that they are all everywhere tangent  to  a fixed line $X_1$  in the plane and to a fixed  plane cubic curve $C_1$ which is smooth or, more generally, reduced and not of additive type.

Thirdly, we show that $\frak N_2(0) = \frak N'_2(0)$, using a limiting argument.

In the second section we shall first prove assertion (I), showing that $\frak N_2$ is dense in   the irreducible subvariety $\frak N'_2$ consisting of the surfaces with  equations in normal form  \eqref{specialeqs}.  The subvariety 
$\frak N'_2$ is irreducible of dimension 15,  contains a Zariski open set where   $K_S$ is ample, and  we show,
 for $S$ general, for instance if $K_S$ is ample (that is, $S$ is equal to the canonical model $X$), the equivalence of the existence of such a free genus $2$ pencil with the 
existence of  a `special' quadratic pencil of conics, satisfying Assumption \ref{**} (see Definition \ref{tangent}).

To show preliminarily that $\overline{\frak N_2} = \frak N'_2$ it suffices to show that such a pencil of conics exists for each  choice of 
a line and of a reduced cubic not of additive type.

 To explain the relation between the first part of the paper and the second, let us consider the following
elementary example: if we take a smooth quadric $\sQ \subset \PP^3$, then it is well known
that $\sQ \cong \PP^1 \times \PP^1$, hence $\sQ$ admits two fibrations onto $\PP^1$.

On the other hand, projection with centre a point $O \in \PP^3, \ O \notin \sQ,$ yields a 
finite morphism $\psi \colon \sQ \ra \PP^2$, a double covering branched on a smooth conic $\sC$
( $\sQ = \{ z^2 = x^2 + y^2 + 1\}$ in affine coordinates).

The fibres of one fibration, which were disjoint lines in $\PP^3$, map to lines in $\PP^2$,
which are necessarily no longer disjoint. How can this happen? 

This is simple to see in this example: the lines $F(j)_t$ of one fibration  map to the lines $L_t$
in the plane tangent to $\sC$ at the point $t$. Since $L_t$ is tangent to $\sC$,
its inverse image in $\sQ$ splits as the sum $F(1)_t + F(2)_t$, whereas the inverse image 
of the intersection point $ P = L_t \cap L_s$ consists of the two intersection
points  $F(1)_t \cap  F(2)_s, F(2)_t \cap  F(1)_s$.
 
Hence the main philosophy is that a fibration can be detected downstairs from the existence 
of a family of curves everywhere tangent to the branch curve. In this article we have a family of conics
everywhere tangent to a curve which is the union of a cubic and a line.

The second part of the article  is devoted to the problem of finding explicit  equations for the variety $\sT$ consisting of the triples: cubic, line, and  quadratic pencil $Q_{\la}$ of  generically smooth everywhere tangent conics. As stated in Theorem \ref{main}, this variety is a covering of degree $12$ of the variety consisting of 
 the pairs: $(C_1, X_1)$, where $C_1$ is a reduced cubic not of additive type, and $X_1$
 is a line not contained in $C_1$.

 Rather than giving very complicate equations for $\sT$, we show that the other two projections are birational,
 and produce explicitly the inverse maps.

Namely, in the first section of  Part II   we shall show the existence of  quadratic pencils of conics 
$Q_{\la}$ everywhere tangent to a fixed line,
and shall see, using the theory of envelopes,  that the conics  are also tangent to a plane cubic,
whose equation will  be explicitly given.

In the  following   section we shall  start instead  from any fixed smooth cubic curve $C_1$ and study quadratic pencils of conics 
$Q_\la$ which are everywhere tangent to the  smooth plane cubic curve $C_1$, showing, again by the theory of envelopes, 
 that there is a line $X_1$ which is tangent to 
each $Q_\la$. We also obtain a theorem which could be classical, identifying the variety of rank 2 conics 
everywhere tangent to a smooth cubic curve $C_1$.

This  second approach  leads also to an explicit construction of 
 the family of  conics 
$Q_\la$ starting  from the Weierstrass normal form of the equation of $C_1$, from the choice of a non-trivial 2-torsion 
line bundle $\eta \in Pic^0(C_1)$, and the choice of a linear pencil (a line)  in the  linear system $|H + \eta|$, where
 $H$ is the  
degree $3$ divisor on $C_1$ giving the embedding $C_1 \subset \PP^2$.

\section{First steps of the proof of the main Theorem \ref{main}: equivalence of the  existence of genus $2$ pencils to a question about quadratic pencils of conics everywhere tangent to a line and a cubic.}

\begin{proof}

 Let us first  prove  assertion (II).

Let $C$ be an irreducible  curve on $S$ of geometric genus 2 such that $|C|$ has dimension at least 1 and contains a pencil $C_t$
of curves of geometric genus 2. Then, if $C^2 =0$, we have a free pencil of curves of genus 2, whose general element is smooth.

Assume now that $C^2 > 0$. Observe that, by the index theorem, $ K_S \cdot C \geq 2$, otherwise
$C$ would be homologous to $K_S$.
Since $S$ is simply connected,  $C$ would be linearly equivalent 
to $K_S$, contradicting that $dim |C| \geq 1.$ 

Let $\pi : T \ra S$ be a birational modification such that the moving parts $\tilde{C}_t$ of the inverse images of the 
curves $C_t$ in the pencil yield  a free pencil of genus 2. Then $\tilde{C}_t^2 =0$, $K_T \tilde{C}_t=2$.

Since $C^2 >0$, we have an exceptional divisor $E$ such that

$$ 2 =   K_T \tilde{C}_t= (\pi^* (K_S) + E) \tilde{C}_t > K_S \cdot C \geq 2,$$
a contradiction.

{\em Let us now pass to proving  the more complicated items (I) and (III).}

Let $X$ be the canonical model of $S$: then $X$ is a weighted complete intersection of type $(6,6)$ in the 
 weighted projective space $\PP (1, 2,2,3, 3)$, with coordinates 
$$ (v, x_2,  x_3, z_3, z_4), \ {\rm with \ respective \ weights} \ (1,2,2,3,3).$$

This was proven in \cite{cat78} (see also \cite{cat80}), and it was shown that the bicanonical map 
$$ \phi : = \phi_2 : S \ra \PP^2 $$ is a degree $4$ morphism given by 
$ \phi (v, x_2,  x_3, z_3, z_4) = (v^2, x_2,  x_3)$.

Set for convenience $ x_1 : = v^2, x : = (x_1, x_2, x_3)$, and observe that to be a  weighted complete intersection of type $(6,6)$
means that 
$$ X : = \{ (v, x_2,  x_3, z_3, z_4) |  f_1 (v, x_2,  x_3, z_3, z_4) = f_2 (v, x_2,  x_3, z_3, z_4) =0)\},$$
where $f_1, f_2$ are weighted homogeneous of degree $6$.

We can write 
\begin{equation}\label{twoeqs}f_j (v, x_2,  x_3, z_3, z_4) = Q_j ( z_3, z_4) + a_j (x) v z_3 + b_j (x) v z_4 + c_j(x),
\end{equation}
with $Q_j$ a quadratic form, $a_j, b_j$ linear forms, and $c_j$ a cubic form.

The two quadratic forms $Q_1, Q_2$ must be linearly independent, else, by taking a linear combination of
$f_1, f_2$  we can assume that $Q_2 \equiv 0$, and then the degree of $\phi$ would not be $4$, but $2$, as we 
have, for fixed $x$, one quadratic and one linear equation for $(z_3, z_4)$: this is a contradiction.

{\bf Assume now that $S$ admits a (free) genus $2$ pencil}: 

then the hyperelliptic involution on the fibres induces a
birational automorphism $\iota$  of $S$, which is biregular since $S$ is minimal;  we get  a biregular  involution  also on the
canonical model $X$, which we continue to denote by  $\iota$.

\medskip

Since $\iota$ is a biregular involution on $S$,  $ Fix(\iota)$ consists of a finite union of smooth curves, 
and of isolated fixpoints $p$, such that the  action  of $\iota$ on the tangent space at $p$ equals  multiplication by $-1$.

In particular, $ \Sigma  : = S /  \iota$ is a normal surface whose only singularities are $A_1$-singularities, called nodes, corresponding to
the isolated fixpoints of   $\iota$. And $\Sigma$ dominates via a birational morphism  the normal surface $Y : = X / \iota$, whose singularities are 
either ADE singularities (also called RDP's = Rational Double Points) or quotients of these by an involution
 (see e.g. \cite{autRDP} pages 10 and following for a list).

{\bf Observe, for later use, that the singularities of $Y$ are rational singularities, hence they are RDP's if they
are hypersurface double points.}

Moreover, clearly the  bicanonical map $\phi$ factors through $Y$, hence $\Psi$ induces a  rational map
$ \psi : Y \dashrightarrow \PP^1$.

More  importantly, $Y$ is covered by the quotients of the genus 2 fibres,  which are rational curves, hence $Y$ is a ruled surface,
 so it has geometric genus zero, whence there are no nonzero  $\iota$-invariant  sections in $H^0(\hol_S(K_S))$:
 in our case, this means that $v$ is an eigenvector for $\iota^*$ with eigenvalue $-1$.

To simplify some formulae, instead of looking at  the action of $\iota$ on the canonical ring, we look  at its action
on the weighted projective space $\PP(1,2,2,3,3)$ and we may  assume that  $v \mapsto v$,
composing $\iota$ with the 
 weighted $\CC^*$-action  given by multiplication by $-1$.

\bigskip

 Since $x_1, x_2, x_3$ are eigenvectors with eigenvalue $+1$, we can assume that also $z_3, z_4$ are 
eigenvectors. And they cannot both have eigenvalue $+1$, since otherwise $\iota$ would  act as the identity.

 Hence there are two cases:

 Case 1): both  $z_3, z_4$ have eigenvalue $-1$.

 Case 2): $z_3$ has eigenvalue $+1$, $z_4$ has  eigenvalue $-1$.

\medskip

{\bf Case 1):} By a linear change of variables involving $z_3, z_4$ we can obtain as in \cite{cat78} that $Q_1 = z_3^2$, $Q_2 = z_4^2$.

In this  case the polynomials $f_j$ of \eqref{twoeqs} are necessarily  eigenvectors with eigenvalue $+1$, hence, since we assumed $v \mapsto v$,  the polynomials $a_j, b_j$ are identically zero
and we have the equations

$$ z_3^2  +  c_1(x) =0, \ z_4^2  +  c_2(x) =0.$$
These equations  exhibit  $X \ra \PP^2$ as a finite Galois covering with group $(\ZZ/2)^2$, a so-called bidouble covering.

We have two intermediate double coverings $Z_j : = \{ w_j^2 = x_1 c_j\}$, where $w_j : = v z_j$, and the 
third intermediate quotient $ Y := X / \iota$, defined, setting $ u := z_3 z_4$ and observing that $u$ is invariant,
  by 
$$ Y = X / \iota = \{ (x,u) | u^2 = c_1 c_2(x)\}.$$

The singularities of $Y$ are double point hypersurface singularities, which are rational,
hence they are Rational Double Points (see Theorems 2.2, 2.4, 2.5 of \cite{autRDP}, showing that the non RDP quotients 
of RDP's by an involution have local embedding dimension $4$ or $5$.)

Since the branch locus of $Y \ra 
\PP^2$ has degree $6$, $Y$ is a  K3 surface with Rational Double Points,  contradicting that $Y$ should be ruled.

 Hence case 1) is excluded.

\medskip

{\bf Case 2):} 

We can assume that $f_1, f_2$, respectively $Q_1, Q_2,$  are eigenvectors for $\iota$, with $f_j, Q_j$
having the same eigenvalue.

For eigenvalue $-1$ the only quadratic term available  is $z_3 z_4$, hence we may assume 
that $f_1$ has eigenvalue $+1$.

Case 2-): assume that instead $f_2$ has eigenvalue $-1$.

Then we may assume $Q_2 = z_3 z_4$, $Q_1 = \la_3 z_3^2 + \la_4 z_4^2$.

Since $f_2$ is an eigenvector with eigenvalue $-1$ it follows immediately that $f_2$ is divisible by $z_4$,
a contradiction because $f_2$ is irreducible.

\smallskip

Case 2+): We can assume that $f_1, f_2$ have both eigenvalue $+1$.

Hence we may  assume, by  combining linearly the two equations,   that $Q_1 = z_3^2, Q_2 = z_4^2$. 

Moreover  $b_1, b_2$ are then identically zero.

Hence we get the two equations of the form

$$ z_3^2 + a_1 (x) v z_3  + c_1(x)= z_4^2  + a_2 (x) v z_3 + c_2(x)=0. $$

Completing the square for the first equation, we may assume that $a_1$ is identically zero, 
hence the equations  have now the form

\begin{equation}\label{specialeqs}
z_3^2 + c_1(x)= z_4^2  + a_2 (x) v z_3 + c_2(x)=0. 
\end{equation}

These equations exhibit  $X$ as an iterated double covering of $\PP^2$, and since $z_3$ has  
$\iota$- eigenvalue $1$, we see, setting $ w : = v z_3$,  that 
$$Y : = \{ w^2 +  x_1 c_1(x) =0\},$$
namely, $Y$ is a normal Del Pezzo surface of degree $2$ with RDP's as singularities (by the same token as before)  and $\phi_Y : Y \ra \PP^2 $ 
is  ramified on the quartic curve union of a line $X_1 = \{x_1=0\}$ and a cubic $C_1= \{c_1=0\}$.

As such, $Y$ is ruled.

Moreover, $Y$ has singular points where $ w= x_1= c_1(x)=0$, or over the singular points of $C_1 : = \{ x | c_1(x)=0\}$.

\bigskip

{\bf Observation:} If  $Y$ is nodal (this holds true if $S=X$, that is, $K_S$ is ample)   $ \{ x | x_1 = c_1(x)=0\}$  consists of three distinct points.

\bigskip

Looking then at the double covering $X \ra Y$ we see that it is branched on the intersection of $Y$ with the cubic
$ \{ w a_2(x) + c_2(x)=0 \}$ (i.e., this is a  divisor in $|-3K_Y|$).

 Hence straightforward calculation yields the following:

  \begin{lemma}\label{ample}
   Let $\frak N'_2$ be the subvariety of the moduli space consisting of surfaces whose equations can be 
   put in the normal form 
   \eqref{specialeqs}. 
   
   Then the canonical model $X$ of a surface $S$  in $\frak N'_2$ is smooth 
   (equivalently, $S=X$ and $K_S$ is ample) if:
   
   (1)   $  C_1= \{c_1(x)=0\}$ is a smooth plane cubic,
   
   (2) the intersection of the line $X_1$ with $C_1$ is transversal, that is, 
   $ \{ x_1= c_1(x)=0\}$ consists of 3 distinct points,
   
   (3)  the intersection of $Y$ with the cubic
$ \{ w a_2(x) + c_2(x)=0 \}$ is smooth and does not contain the points  $  x_1= c_1(x)=0$ of $Y$.
    \end{lemma}

   \begin{lemma}
   Let $\frak N'_2$ be the subvariety of the moduli space consisting of surfaces whose equations can be 
   put in the normal form 
   \eqref{specialeqs}. Then $\frak N'_2$ is irreducible of dimension $15$.
    \end{lemma} 
   \begin{proof}
   The equations are parametrized by an open set of the affine space of dimension $23$ whose elements $(c_1, c_2, a_2)$ are 
   a pair of cubic equations and a linear equation in $x = (x_1, x_2, x_3)$.
   
   But we have to divide by the group of matrices  $\ga \in GL(3)$ inducing transformations of the plane preserving the line $x_1$, and also allow multiplication of $z_3, z_4$ by nonzero constants $d_3, d_4$ 
   (products of the above yield all the projectivities in $\PP(1,2,2,3,3)$ respecting the given normal forms). 
  
  We get  a group of dimension $9$, which however  contains the $\CC^*$-action. 
   Hence we get dimension $ 23- (9-1) = 15$.
   
   Another way to see it is that the 3-nodal Del Pezzo surface $Y$ has 3 moduli, and the branch locus of the double covering
   $ X \ra Y$ belongs to a  projective space of dimension $12$. Since the group of automorphisms of $Y$ is finite, 
   we conclude that $\frak N'_2$ has dimension $15= 3 + 12$.
   
   \end{proof}

   \subsection{Geometry of the involution $\iota$}
   The fixed locus of $\iota$ intersects the general fibre in the $6$ Weierstrass points,
hence the fixed curve of $\iota$ contains a $6$-section. 

The canonical curve $ \{ v=0\}$ can be written as $D+ E$, where $D$ is  irreducible 
with $K_S \cdot D=1$, and $K_S \cdot E=0$, so that $E$ is a sum of $(-2)$-curves. 

We denote by $D'$ the image of $D$ in the canonical model $X$.

Clearly $D, D'$ are   left   invariant by $\iota$.

If  $D^2 >0$, then, by the index theorem and since $S$ is 1-connected, $D=K_S$ and $ E=0$.
 
  If instead $D^2< 0$, since $D^2 + 1 \geq -2$,
 $D^2 \in \{-1,-3\}$, and in the first case $D$ has arithmetic genus $1$, while in the second case $D$ is smooth of genus $0$, hence there are two fixpoints of $\iota$ on $D$.

There are a priori two cases:

\begin{itemize}
\item
(Hor) $D$ is   horizontal, that is, $D$ is not contained 
 in a fibre $F$ of $\Psi$.
 \item
 (Ver) $D$ is vertical, that is, it is contained in a fibre $F$ of $\Psi$: in this case the sum of the   horizontal components
 of $E$ yields an effective  divisor  $E'$ with $E' \leq E$, such that $E' \cdot F = 2$.
\end{itemize}

If we have  a fixpoint $p$ of $\iota$, which  is a smooth point of  the invariant canonical curve $D+E$, 
there are   local holomorphic coordinates $(y_1, y_2)$  
such that, locally at $p$, $ D+E = \{ y_2=0\}$, $p = (0,0)$ and $ \iota (y_1, y_2) = (\e y_1, -y_2)$, with $\e = \pm 1$.

 Then the holomorphic 2-form $v$ can be written as:
$$ v =  y_2 U(y_1, y_2) dy_1 \wedge dy_2,$$
where $U$ is a unit in the local ring at $p$. By looking at the lowest degree term we see that $v$ is an eigenvector with eigenvalue $\e$.
Hence $\e= - 1$ and the involution is not the identity on the component of $D+E$ containing $p$.

This implies that any invariant component of $D+E$ which is horizontal must intersect the general  fibre in $2$
points, belonging to the hyperelliptic series, which 
 are exchanged by $\iota$. 
 
 \begin{rem}
 We shall consider  from now on the horizontal case, which certainly occurs if  $S=X$,
  which is the general case by virtue of Lemma \ref{ample}.

 Concerning  the vertical case, we consider the relative canonical model $X_{rel}$ of the fibration $ \psi 
 \colon S \ra \PP^1$, where we contract all the $-2$-curves which are vertical.
 
 As in the main theorem 4.7, page 1023, of \cite{c-p}, $X_{rel}$ is a finite double cover of a conic bundle $\sC \ra \PP^1$,
 with RDP's as singularities, $\sC \subset \PP$, and the branch locus intersects the $\PP^2$-fibres of $\PP \ra \PP^1$
 in  a relative cubic curve not containing 
 the singular points of $\sC$.
 
  It follows then that, if $D$ is vertical, $D$ must have arithmetic genus $1$,
 $D$ is  contained in a fibre $D+C$ where $D \cdot C=1$, and also $C$ has arithmetic genus $1$.
 
 Moreover, then the horizontal curve $E'$ must be irreducible and meets the general fibre in two distinct points.
 
 Since  $E'$ contracts to a point $P$ of $X$, this  would make all the images of the fibres 
 to pass through $P$, as well as the canonical curve $D'$.  Then their images in the plane under the morphism 
  $\psi$ would  be, as we shall show,  conics in the plane. And they will all be 
  passing through the same point of the line $X_1 = \{ x_1=0\}$ which is the  image of the canonical curve $D'$.

We believe it unlikely that this case occurs, but it would take long 
 to  exclude it.
 \end{rem}

\subsection{How to obtain the pencil of conics $Q_{\la}$}

The genus $2$ pencil $\psi : X \ra \PP^1$,
with general fibre $F$, induces a fibration $ \psi_Y : Y  \dashrightarrow  \PP^1$, with general fibre $L \cong \PP^1$.

Assume that we are in the horizontal case, where $D$ is not contained in a fibre:
then $ D \cdot F = 2$, and the image $D^0$ of $D$ in $Y$
is a section of $\psi_Y$,  which is now a morphism.

Observe  that,  if $\phi^*(R)$ is the inverse image of a line $R$  in $\PP^2$, then $\phi^*(R) \equiv 2 K_S$.

Therefore, using that $ \phi_* (F) = 2 (\phi_Y)_*(L) $,  we infer that  $(\phi_Y)_*(L) $, which is a priori   either a line $R$ or a conic $Q$,
cannot be a line:   else, the inverse image of these lines $R$ yield:
$ 2 K_S \equiv \phi^* (R) = F + M$, with $M \in | 2 K_S-F|$,
and we get a contradiction from
$$ 2 = 2 K_S ^2 = K_S (F + M)= 2 +  K_S  \cdot M,$$
an equality which implies that the divisors $M$ cannot move.

Therefore $Q$ is a conic, in general irreducible, and it must be  everywhere tangent to the branch locus,
so that $(\phi_Y)^*(Q)$ splits as $L  + j(L)$, where $j$ is the involution of $Y$ induced by the
 double covering $\phi_Y$. The family of curves $L$, $Q$ is parametrized by $\la \in D^0$, since $D^0$ is a section,
  hence we shall denote them $Q_\la$, with $\la \in \PP^1 \cong D^0$.
In particular $Q_\la$ is tangent to $X_1 \cong D^0$ in the point $\la$.

\begin{lemma}\label{quadratic-pencil}
Assume that we have a 1-parameter family of conics  $Q_\la$ in the plane, parametrized by $\la \in \PP^1$
such that there exists a line $X_1\subset  \PP^2$ and an isomorphism $ X_1 \cong \PP^1$ 
such that for each point $\la \in X_1$ there is 
exactly one conic $Q_\la$  tangent to $X_1$ in the point $\la$.

Then $\{ Q_\la\}$ is a quadratic pencil, that is, we have a conic $\sQ$  in the space of conics, $\cong \PP^5$.
\end{lemma}
   
   \begin{proof}
   
  Take  coordinates $x = (x_1, x_2,x_3)$ on the plane such that  a point $\la \in X_1$
    has coordinates $(0, \la_2, \la_3)$. 
    
    Then   the conic $Q_\la$ has then equation of the form
    $$ Q_\la = (\la_3 x_2 -  \la_2 x_3)^2 + 2 x_1 (\la_2^2 M_2(x) + \la_3^2 M_3(x)+ \la_2 \la_3 M_1(x) )=0,$$ 
    where $M_j = \sum_i m_j^ix_i$ is a linear form.

   \end{proof} 
 
    If the cubic $C_1$ is smooth,  $Q_\la$ cuts twice a  divisor $\de_\la \in |H + \eta|$, where $\eta$ is a 2-torsion divisor,
   and where $H$ is the divisor on $C_1 \subset \PP^2$ cut by a line.  The divisor $\eta$ has torsion order precisely  two. Otherwise
   $Q_\la$ would always be a line with multiplicity two, a contradiction.
  
  \begin{remark}

   Whence, the conic  $Q_\la$ is never a double line, but  it can consist of two lines:  
take for instance  a flexpoint $O$ as the origin of the elliptic curve $C_1$ and let $p_1, p_2$ be 
distinct 2-torsion points:
then  $|H|$ contains the divisors $O + 2 p_1$ and $O + 2 p_2$, whose sum is the double of  $ O + p_1+ p_2 \in |H + \eta|, \  \eta : = p_1 -p_2$.

If however the conic  $Q_\la$ consists of two lines, being tangent to the line $X_1$ is only possible if the double point
of $Q_\la$ lies on $X_1$.

\end{remark}

{\bf Conclusion:}  the only conics $Q_\lambda$ which consist of two lines are such that  $\la \in X_1 \cap C_1$, hence they are
at most  $3$.
 In fact, for $\la \in X_1 \cap C_1$,  since $Q_\la$ is  also everywhere tangent to $C_1$, if $C_1$  is transversal to the line $X_1$,
then $Q_\la$ has multiplicity $2$ at $\la$, hence consists of two lines.

\end{proof}

   There remains to see that  we can find a family of conics $Q_\la$ with the desired properties.
   
\begin{defin}\label{tangent}
Given a line $X_1$ and a cubic curve $C_1$ in the plane, 
not containing it, we shall say that a quadratic pencil of conics
(it is necessarily quadratic, by Lemma \ref{quadratic-pencil})
is a {\bf special everywhere tangent pencil of conics}  if the pencil is  
everywhere tangent to the two curves in a special way, i.e.,   it satisfies the following:
 \begin{assumption}\label{**} 
   (i)  $Q_\la$ is tangent to the line $X_1$ at the point $\la \in X_1$ and $Q_\la$  is generically smooth,
   and does never have rank equal to $1$;
   
   (ii) $Q_\la$ is also everywhere tangent to the cubic $C_1$.
 \end{assumption}  
\end{defin}
   
  Let  us illustrate again  the situation:
   the fact that $X$ admits a genus 2 fibration onto $X_1 \cong \PP^1$ is equivalent to the condition that the
   bicanonical map is an iterated double covering $ X  \ra Y  \ra \PP^2$ 
   where 
   $$Y : = \{ w^2 +  x_1 c_1(x) =0\}$$
 is a normal Del Pezzo surface of degree $2$ with RDP's as singularities, and $\phi_Y : Y \ra \PP^2 $ 
is  ramified on the reduced quartic curve $\sE $ union of a line $X_1 $ and a cubic $C_1$,
and $\sE$ is everywhere tangent to  a 
pencil of conics in the plane.

\section{End of the proof of the main Theorem \ref{main}: existence of the special pencil of conics}

We have shown that $\frak N_2 \subset \frak N'_2$, and that $\frak N'_2$ is an irreducible subvariety of dimension $15$.

 We first   show that $\overline{\frak N_2 }= \frak N'_2$: for this, it suffices to show that the general point in $ \frak N'_2$
is contained in $\frak N_2$. Hence, we only need to show that, for a general pair $X_1, C_1$ of a line and a smooth cubic, there is
a pencil of conics satisfying Assumption \ref{**} of Definition \ref{tangent}.

But, as stated in the theorem, we can prove indeed more: that this holds under the assumption that $X_1$ is not contained in $C_1$ and that $C_1$ is not of additive type.

\bigskip 
Let $C_1$ be a reduced  cubic curve in $\PP^2$ and $X_1$ a line which is not a component of  $C_1$.

Let $B : = C_1 \cup X_1$. The quartic $B$ has arithmetic genus equal to 3, and by virtue of the exact sequence
$$ 0 \ra \hol_B \ra ( \hol_{C_1} \oplus \hol_{X_1} ) \ra \hol_{X_1 \cap C_1} \ra 0$$
and the similar exact sequence for $\hol^*$, we get the associated long exact cohomology sequence,
since $Pic(X_1) = \ZZ$: 
$$ ( Pic) \ 1 \ra \CC^* \ra H^0(\hol^*_{X_1 \cap C_1}) \ra Pic(B) \ra Pic (C_1) \oplus \ZZ \ra1.
$$

The exact sequence (Pic) shows that in $Pic(B)$ there are  line bundles $\sL$ such that 
$\sL ^2 \cong \hol_B(2)$, and these restrict to $(H + \eta, 1)$, where $ 2  \eta \equiv 0$ on $C_1$.

However, in order that Assumption \ref{**} is satisfied, we should have that the restriction of $\sL$ on $C_1$
is of the form $\hol_{C_1}(H + \eta)$,
with $\eta \in Pic(C_1)_2$ nontrivial.

For this condition we need that $Pic^0(C_1)$ is not an additive group $\cong \CC$, or, as one says, $C_1$
 is not of additive type.

This requires that we avoid the Kodaira types $II, III, IV$, namely, that $C_1$ does not consist of a cuspidal cubic,
or a line tangent to a smooth conic, or three lines through a point.

In all other singular cases $\CC^* = Pic^0(C_1)$, hence the desired $\eta$ can be found
if $C_1$ is not of additive type.

The exact sequence (Pic)  specializes to 

$$ 1 \ra \CC^* \ra (\CC^*)^3 \ra Pic(B) \ra Pic (C_1) \oplus \ZZ \ra1$$
in the case where $X_1 \cap C_1$ consists of $3$ points.

This exact sequence shows that, if moreover $C_1$ is smooth, then in $Pic(B)$ there are 16 line bundles $\sL_\e$ such that 
$\sL_\e ^2 \cong \hol_B(2)$: they restrict to $\hol_{X_1}(1)$, respectively to $\hol_{C_1}(H + \eta)$,
with $\eta \in Pic(C_1)_2$, and moreover $\e = (\e_1, \e_2,\e_3)$ is a gluing factor such that $\e_i = \pm 1$.

That is, we have the exact sequence
$$ 0 \ra \sL_\e  \ra ( \hol_{C_1}(H + \eta) \oplus \hol_{X_1} (1) ) \ra \hol_{X_1 \cap C_1} \ra 0,$$
and  the associated sequence of global sections
$$ 0 \ra H^0(\sL_\e ) \ra ( H^0(\hol_{C_1}(H + \eta)) \oplus H^0(\hol_{X_1} (1)) ) \ra H^0(\hol_{X_1 \cap C_1}) = \CC^3 $$
is exact if $\eta$ is a nontrivial 2-torsion divisor.

Because then $H^0(\hol_{C_1}( \eta))=0$ and $H^0(\hol_{C_1}(H + \eta)) \ra H^0(\hol_{X_1 \cap C_1})$ is
an isomorphism, and then $H^0(\sL_\e )\cong  H^0(\hol_{X_1} (1)) $.
 
 Hence the sections $\la$ of $H^0(\hol_{X_1} (1)) $ yield unique sections $\s_{\la} $ of $H^0(\sL_\e )$ whose square is the zero set
 of the intersection of a conic $Q_{\la} $ with $B$, and such that the intersection of $Q_{\la} $
 with $X_1$  is the point $\la$ counted with multiplicity two.
 
 This shows that the  required family $Q_{\la} $  of conics exists for all 
 pairs of a reduced  cubic $C_1$ and of  a line $X_1$ not contained in $C_1$, if $C_1$ is not of additive type.

 And that, in the case where the  cubic $C_1$ is smooth  and   intersects transversally the line $X_1$, we have exactly 12 line bundles such that Assumption \ref{**} is satisfied.
 
 To finish the argument, we just observe that the family $Q_{\la} $  of conics yields 
 two  families of lines in the Del Pezzo surface $Y$, because the inverse image of  $Q_{\la} $
splits as the union of two rational curves meeting in $4$ points, each having  intersection number 2 with the anticanonical divisor  of $Y$. In turn, when we take the double covering $X \ra Y$, the inverse image of one of these lines
is a double covering of a line branched in $6$ points by equation \eqref{specialeqs}, hence a genus $2$ curve.

 Hence for each pencil of conics we obtain two pencils of genus $2$ curves, which are distinct, since their intersection product is nonzero. Thus we get $ 12 \times 2 = 24$ such genus two pencils.

 \subsection{ End of the proof of Theorem \ref{main}: $\frak N_2  (0) = \frak N'_2 (0)$.}
 
In view of the density $\overline{\frak N_2 }= \frak N'_2$ it suffices to show that $\frak N_2 (0)$ is closed 
 in $\frak N'_2 (0)$, via 1-parameter
families. 

We do it then:

\begin{theorem}
Assume that we have a family of $(1,1)$-surfaces $$\pi :  \sS \ra T,$$ where $T$ is a connected smooth 
 contractible complex curve, with $0 \in T$ and, 
for $ t \neq 0$, $S_t$ admits a genus two pencil $F_t$. 

Then, if $S_0$ has ample canonical divisor, then  
 also $S_0$ admits a genus $2$
pencil $F_0$.
\end{theorem}

\begin{proof}
For each $S= S_t$ the exponential sequence yields an exact sequence,
identifying a divisor class as a cohomology class of type $(1,1)$
(theorem of Lefschetz on divisors):

$$ 0 \ra  H^1(\hol^*_S) \ra H^2(S, \ZZ) \xrightarrow{\de}  H^2(\hol_S) .$$

Since  $T$ is contractible, we have a differentiably trivial fibre bundle. Moreover, by Baire' s Theorem, we can
assume that    we have a section of the relative Picard scheme, such that  the class of $F_t$
in  $H^2(S_t, \ZZ) $ is constant. Then this class  extends to $F_0$, and we claim that
its image inside $H^2(\hol_{S_0})$ is zero.

This follows from the base change theorem (see \cite{av}), saying that, since the cohomology groups
 $H^j(\hol_{S_t})$ have constant dimension, the
direct image $\sR^2\pi_* (\hol_{\sS})$ is a line bundle on $T$.

We have a holomorphic section which vanishes for $ t \neq0$, hence its limit is also zero.

\bigskip

Therefore we have found a divisor class $F_0$ on $S_0$ such that $$F_0^2 =0 , K_{S_0} \cdot F_0 = 2.$$ 

By semicontinuity, since $ h^0( \hol_{S_t}(F_t))=2$ for $t \neq 0$, we get 
$$ h^0( \hol_{S_0}(F_0))\geq 2.$$

Let us prove that $S : = S_0$ has a genus $2$ pencil. Set $K : = K_S$.

Then we have that the linear system $F_0$ can be written as the sum of a movable linear system 
an a fixed part:

$$ |F_0| = |M| + \Phi.$$

Since the surfaces $S$ is of general type, the genus of the system $M$ is at least $2$, and if indeed 
$KM + M^2 = 2$, then we have found the desired genus $2$ pencil.

Otherwise, we must have $KM + M^2 \geq 4$.

Since $K$ is nef, $ K F_0 =2$ implies $KM + K\Phi =2$. 

Since $M$ is movable,
we have $ KM \geq 1$, hence there are only the two cases $KM=1$ or $KM=2$.

The index theorem says that 
\begin{equation}\label{index}M^2 \leq  (MK)^2 \end{equation}

Then equation \eqref{index} yields: for $KM=1$, $M^2 \leq 1$ , contradicting $KM + M^2 \geq 4$.

For $KM=2$ from $ K F_0 = 2$ we get  $K \Phi=0$, and if we have a fixed part, it consists of $-2$-curves.

This is a contradiction to the assumption that the canonical divisor $K$ of $S_0$ is ample.

\end{proof}

\section{Part II : The theory of envelopes 
and explicit equations for special pencils of conics everywhere tangent to a line and a cubic.}

 We have shown in Part I  that, given a  cubic $C_1$ which is smooth, or just not of additive type,
 and a line $X_1$ not contained in $C_1$, 
 there is a quadratic pencil of conics $Q_{
 \la}$  which
are everywhere tangent to $C_1$, and also to  $X_1$, and satisfy Assumption \ref{**}. 

Our goal here is to give somehow explicit equations for the variety $\sT$ of such triples $(C_1, X_1, \{ Q_{
 \la}\})$, at least from the birational point of view.
 
 Our main Theorem \ref{main} says that this variety $\sT$  is  a covering of generic degree $12$ of the variety $\sB$
 of the reduced quartics $ B = C_1 \cup X_1$.
 
 The interesting fact that we shall now show is that the other two projections, of $\sT$ to the variety $\sP_L$ of
 pairs $(X_1,  \{ Q_{\la}\})$ satisfying the tangency condition of Assumption \ref{**}, as well as the projection onto 
 the variety $\sP_{C_1}$ of
 pairs $(C_1,  \{ Q_{\la}\})$ satisfying the tangency condition of Assumption \ref{**}, are both birational.
 
 These birationality statements are proven in Proposition \ref{residual-cubic}, respectively in Theorem \ref{weierstrass}.
 
 Proving this birationality statement will amount to giving explicit equations, in the  former case  for the cubic $C_1$ in terms of  $(X_1,  \{ Q_{\la}\})$, in the latter case 
  for the line $X_1$ in terms of $(C_1,  \{ Q_{\la}\})$.
 
 This will be done in the next  sections using the theory of envelopes, that is, curves tangent to all curves in a given 1-parameter family.
 
 In between, we need also somehow to describe our variety $\sT$, showing how the pencils 
 $ \{ Q_{\la}\}$, which   are indeed conics in the five-dimensional projective space $\PP^5$ of plane conics,
 are associated to a cubic and a nontrivial 2-torsion line bundle $\eta$ on it.

We  show  in the next section  how, fixing a line $X_1$ and 3 distinct points in  it, there is a 6-dimensional family of cubics $C_1$ passing through the 3 points, obtained each as an everywhere tangent cubic 
to  a quadratic pencil of conics  tangent 
to $X_1$.

   \section{Explicit equations  of 
special everywhere tangent pencils of conics starting  from 
 a line $X_1$ and   three  points on it: the pencil determines the cubic curve $C_1$.}

In this section, we shall fix $X_1$ as the line $x_1=0$, and the three points $X_1 \cap C_1$: we shall see that 
Assumption \ref{**}
implies that the equation of $C_1$ can be written as 

\begin{equation}\label{cubic}
C_1  =  \{ x|   [ 2 M_2(x)  x_2^2 +2  M_3(x)x_3^2 + 4 x_1M_2(x) M_3(x) + 2 x_2 x_3  M_1(x)- x_1M_1(x)^2]=0\},
\end{equation}
where $M_1, M_2, M_3$ are linear forms.

    Take the  coordinates $x = (x_1, x_2,x_3)$ on the plane previously considered in Lemma \ref{quadratic-pencil}: then a point $\la \in X_1$
    has coordinates $(0, \la_2, \la_3)$, and the conic $Q_\la$ has  equation of the form
    $$ Q_\la = (\la_3 x_2 -  \la_2 x_3)^2 + 2 x_1 (\la_2^2 M_2(x) + \la_3^2 M_3(x)+ \la_2 \la_3 M_1(x) )=0,$$ 
    where $M_j = \sum_i m_j^ix_i$ is a linear form.

    We have a map of $\PP^1$ with coordinate $\la$ to the projective space $|H + \eta|$ (this system
    yields another embedding of $C_1$ in the plane); $\la \mapsto \de_\la$, and $2 \de_\la \in |2H|$,
    where $|2H|$ is the space of conics in $\PP^2$. Therefore, the composite being quadratic, the map
    $\la \mapsto \de_\la$ is linear, and we have a linear pencil in $|H + \eta|$.

   Hence we get the desired 1-parameter family of quadratic forms $Q_\la$.

\begin{proposition}\label{residual-cubic}

There is a cubic $C_1$ such that $C_1$ is everywhere tangent to each  $Q_\la$  (that is, tangent at each intersection point) .
\end{proposition}

\begin{proof}
We use  here the theory of envelopes.

\smallskip

We have
 $$ Q_\la = \la_3^2 (x_2^2 +2 x_1 M_3(x))  +  \la_2 \la_3 (-2 x_2 x_3  +2 x_1M_1(x)) +
       \la_2^2 (x_3^2 + 2 x_1 M_2(x)).$$
      
      And we consider now the surface $$\sQ  = \{ (\la, x) | Q_\la (x)=0\} \subset \PP^1 \times \PP^2.$$ 
      
      $\sQ \ra \PP^2$ is a double covering with  branch curve   the  discriminant quartic curve  $\sE$, called  the 
      envelope:
      
      $$\sE : =  \{ x| (x_2^2 +2 x_1 M_3(x))(x_3^2 + 2 x_1 M_2(x)) - ( x_2 x_3  - x_1M_1(x))^2=0 \} \Leftrightarrow$$ 
$$\sE : =  \{ x|  x_1 [ 2 M_2(x)  x_2^2 +2  M_3(x)x_3^2 + 4 x_1M_2(x) M_3(x) + 2 x_2 x_3  M_1(x)- x_1M_1(x)^2]=0 \} .$$ 

We see immediately that $\sE$ consists of the line $X_1$ and of a cubic $C_1$, 
of equation \eqref{cubic} as announced, which is everywhere tangent to
each conic $ Q_\la $, in view of the following Proposition,  and of the fact that if $x \in Q_\la \cap \sE$, 
 then $\la$ is unique with this property: because  $(\la, x) $ belongs to the ramification locus of the double covering.

\end{proof}

\begin{proposition}\label{envelope}
Consider a 1-parameter family of  plane curves $F_\la =0$, for $\la \in \PP^1$,  (i.e., $F_\la (x) = F (\la, x)$) of degree $d$
 such that the general curve is reduced:
 then the {\bf envelope} $\sE$ of the family of curves,
defined as the branch divisor of the second projection $p_2 : \sF \ra \PP^2$,
$$    \sF   : = \{ (\la, x) | F(\la, x)=0\} \subset \PP^1 \times  \PP^2,$$
is  everywhere tangent to the curves $F_\la$ in the family (more precisely, at each smooth point $x \in \sE$  there is 
 exactly one curve $F_\la$ tangent to $\sE$ at $x$).
\end{proposition}
\begin{proof}
Since the general curve is  reduced, $\sF  $ is also reduced.

By adjunction, if $F$ has bidegree $(m,d)$, then  the canonical divisor $K_\sF = \hol_\sF(m-2, d-3)$
and,  by Hurwitz' formula, the ramification divisor $R$ satisfies $\hol_\sF (R) = \hol_\sF(m-2, d)$,
whence the branch locus $\sB$ has 
$$ degree (\sB ) = \sB \cdot H = (m-2, d)\cdot  (m, d)\cdot  (0,1) = d (2m-2).$$

The condition of tangency is  automatic for a curve $F_{\la}$ passing through a  singular point  of $\sB$,
 while for the smooth points $
\xi$ of $\sB$ 
there is exactly one point $(\la',\xi) \in R$ mapping to $x$  (otherwise the multiplicity of  $\sB$ 
at $\xi$ would be at least 2!). And there is nothing to prove if $F_{\la'}$ is singular at $\xi$.

So, we assume that $\xi$ is a smooth point of $F_{\la'}$   and of $\sB$, hence $(\la',\xi)$ is a smooth point of $R$.

Since $R$ is defined by the equations (in affine coordinates)
$$  F(\la, x)=0 , \ \frac{\partial F}{\partial \la} (\la, x)=0,$$
$(\la',\xi)$ is a smooth point of $\sF$.

At  $\xi $ there are local coordinates $(u,v)$ for $\PP^2$, such that $\sB$ equals $ v=0$,
and local coordinates $(u,w)$ on $\sF$  at $(\la',\xi)$ such that locally the map has the form $(u,w) \mapsto (u, w^2)$
 (simple ramification follows again from the assumption that $\xi$ is a smooth point of the branch divisor).

 Hence the derivative of the second projection $ p_2 : \sF \ra \PP^2$ has rank equal to 1 at the point
 $(\la',\xi)$, and its image is the 1-dimensional space generated by the tangent to $\sB$.
 
 Therefore, also the tangent of $F_{\la'}$ is a multiple of the tangent to $\sB$ at $\xi$,
 and our assertion is proven.

\end{proof}

\begin{remark}
 If the curves $F_{\la}$ are all reduced, and the general one is smooth, then $\sF$ is normal.
If all the curves  $F_{\la}$ are singular, their singular points, by Bertini's theorem, yield curves 
inside $ Sing (\sF)$, whose image in $\sB$ have multiplicities $ \geq 2$.

\end{remark}

\subsection{Explicit calculations}
 To do more precise calculations, let us   assume that the three points of intersection of $X_1$ with $C_1$ are 
 $$(0,0,1), (0,1,0), (0,1,1).$$
    
    Hence $Q_\la$ should have  vanishing determinant exactly for $\la_2=0, \la_3=0, \la_2=\la_3$,
    and exactly rank $2$.
    
     For $\la_2=0$ we get from the rank condition that $M_3 (x)= M_{3,2} x_2 + M_{3,1} x_1$,
     and $2 M_{3,1} \neq M_{3,2}^2$. Then $Q_{0,1} = x_2^2 - 2 M_{3,2} x_2 x_1 + 2 M_{3,1} x_1^2=0.$
     
     For  $\la_3=0$ we get from the rank condition that $M_2 (x)= M_{2,3} x_3 + M_{2,1} x_1$,
     and $2 M_{2,1} \neq M_{2,3}^2$. Then $Q_{1,0} = x_3^2 - 2 M_{2,3} x_3 x_1 + 2 M_{2,1} x_1^2=0.$
     
     For $\la_2= \la_3=1$ we get the form
     $$  ( x_2 -   x_3)^2 +  2 x_1 ( M_2(x) +  M_3(x)+ M_1(x) ) = : ( x_2 -   x_3)^2 +  2 x_1 N(x),$$
     and $$N_{2} + N_{3} =0 \Leftrightarrow M_{3,2} + M_{2,3} + M_{1,2}+ M_{1,3}=0,$$
     and $N_1 \neq N_2^2  \Leftrightarrow 2 (M_{2,1} + M_{3,1} + M_{1,1}) \neq (M_{3,2} + M_{1,2})^2$.
     
     Then $Q_{1,1} = (x_2 -  x_3)^2 +  2 N_2   x_1 (x_2 -  x_3) + 2  N_1  x_1^2=0.$

     Hence 
     
      $$ Q_\la = (\la_3 x_2 -  \la_2 x_3)^2 + $$
      $$ +  2 x_1  \big(  \la_2^2 (M_{2,3} x_3 + M_{2,1} x_1)  + \la_3^2 (M_{3,2} x_2 + M_{3,1} x_1) + \la_2 \la_3 (M_{1,3} x_3 + M_{1,2} x_2 + M_{1,1} x_1 )  \big)  .$$
     
    \begin{remark}
    As a verification, we 
    calculate  the determinant of the conic $ Q_\la$, which  we know  a priori  to  be a   degree 6 polynomial with exactly $3$ roots, 
    hence it should  have  some double root.

    We calculate the determinant of the following matrix
    
    \[
		A=\begin{pmatrix}
			\la_3^2 &  - \la_2 \la_3 & \la_3 [\la_3   M_{3,2} + \la_2 M_{1,2}]\\
			- \la_2 \la_3 & \la_2^2  &  \la_2 [[ M_{2,3} + \la_3 M_{1,3}]] \\
			\la_3 [\la_3   M_{3,2} + \la_2 M_{1,2}] &  \la_2 [[  \la_2 M_{2,3} + \la_3 M_{1,3}]] & 2 ( \la_2^2 M_{2,1} +  \la_3^2 M_{3,1} +  \la_2 \la_3  M_{1,1})\\
		\end{pmatrix}
		\]	
    
 Dividing the first row and column by $\la_3$, the second  row and column by $\la_2$,  
 we are left with the matrix
     \[
		B=\begin{pmatrix}
			1 &  - 1&  [\la_3   M_{3,2} + \la_2 M_{1,2}]\\
			- 1  &  1  &  [[ \la_2 M_{2,3} + \la_3 M_{1,3}]] \\
			 [\la_3   M_{3,2} + \la_2 M_{1,2}] &   [[ \la_2 M_{2,3} + \la_3 M_{1,3}]] & 2 ( \la_2^2 M_{2,1} +  \la_3^2 M_{3,1} +  \la_2 \la_3  M_{1,1})\\
		\end{pmatrix}
		\]
		
		Subtracting multiples of the first column, we obtain the matrix
		   \[
		A' =\begin{pmatrix}
			1 &  0 &  0\\
			- 1  &  0  & [\dots ] +  [[ \dots ]] \\
			 [\la_3   M_{3,2} + \la_2 M_{1,2}] &  [\la_3   M_{3,2} + \la_2 M_{1,2}] +  [[ \la_2 M_{2,3} + \la_3 M_{1,3}]] & 2 (\dots)-[\dots]^2		\end{pmatrix}
		\]    
  hence the determinant of $A$ equals 
  $$ det (A) = - \la_2^2 \la_3^2 ( [\dots ] +  [[ \dots ]])^2 =  - \la_2^2 \la_3^2 (\la_3 (M_{3,2} + M_{1,3}) + \la_2 (M_{1,2} + M_{2,3} ) )=$$
    
$$ = - \la_2^2 \la_3^2  (\la_2- \la_3 )^2 (M_{3,2} + M_{1,3}) ^2 .$$

The further condition $(M_{3,2} + M_{1,3})  \neq 0$ ensures that the determinant is not identically  zero.
 
\end{remark}
 
 \bigskip
 
 Theorem \ref{main} shows that via  the above construction, depending on 6 parameters,  we get all the possible  cubic curves
passing through the given 3 points in $X_1$.

\section{Quadratic pencils of conics everywhere tangent to a  plane cubic which is smooth (or just not of additive type)}

In this section we take up again our discussion, beginning now with any  plane cubic $C_1 $ smooth, or  not of additive type, embedded in $\PP^2$ by the complete linear system $|H|$: we want to show that a pencil of conics everywhere tangent to $C_1$ is necessarily also
tangent to a fixed line in the plane.

 We use again the theory of envelopes of 1-parameter families of plane curves, as done in  Proposition \ref{envelope}.

\begin{lemma}\label{line-envelope}
Given a quadratic  pencil $\Lambda$ of conics, that is,  a quadratic 1-parameter  family of conics in the plane,
$$ Q_\la = \{ x |  \sum_{i\leq j} \la_i \la_j Q_{i,j} (x) = 0\}, \ \la \in \PP^1,$$
determining  a  normal  surface $\sQ \subset \PP^1 \times \PP^2$  of bidegree $(2,2)$ (in particular, the general $Q_\la$ is then smooth),
 assume that there is a  reduced 
cubic $C_1$ such that the conics of the pencil are everywhere tangent to $C_1$ (i.e., $C_1$ cuts on the general 
$Q_\la$ a divisor with all multiplicities $\geq 2$):
then there is a line $X_1$ such that each $Q_\la $ is tangent to $X_1$.

\end{lemma}

\begin{proof}
The conics of the pencil, taken  together,  determine a surface $\sQ \subset \PP^1 \times \PP^2$ of bidegree $(2,2)$, a double covering of $\PP^2$ branched on the discriminant 
quartic curve $$\sE : = \{ Q_{1,2}^2 (x)  - 4 Q_{1,1} (x) Q_{2,2} (x) =0\},$$
classically called the envelope of the family of conics, and everywhere  tangent to some conic of the family.
 Observe that the discriminant is nonzero, else $\sQ$ is reducible and twice a pencil of lines.

The singular points of $\sQ$ map onto points $P$ of $\sE$ and, if  such a point $P \in \sE$ is general, $\la$ is unique,
so we may denote it $\la(P)$.
Moreover, we have a singular point of $\sQ$ iff $P$ is a singular point for $Q_{\la(P)}$.

Hence the condition that $\sQ$ is normal implies  that the general $Q_\la$ is smooth.

 For each general point $P$ of $C_1$ there is  one conic $ Q_\la $ passing through  $P$,
and, by our assumption, it must be tangent to $C_1$ at $P$.

We show now that  that  there is only one conic $ Q_\la $ passing through $P$, 
so  that
 $C_1 \subset \sE$, hence  we can write $\sE = C_1 \cup X_1$, where $X_1$ is a line.
 
  Indeed, if there were two distinct conics  $ Q_\la $ tangent to $C_1$ at a general point $P$,
 then the inverse image of $C_1$ is a double covering $C'_1$ of $C_1$ and, because of the tangency condition, the map to $\PP^1$ given by
 $\la$ has zero derivative hence is locally constant on $C'_1$; whence it is constant on each connected component. We conclude that
 $C'_1$ would then be disconnected, a contradiction because $C_1 \cap \sE \neq \emptyset.$

And for each point $ y \in X_1$,
 there is  exactly one  $\la_y$  such that  $y \in Q_{\la_y}$.

The second projection $\sQ \ra \PP^2$ ramifies
 at $(\la_y, y)$, hence the curve $Q_{\la_y} $ is tangent to $X_1$ at $y$ if $y \notin C_1$, as we saw in Proposition \ref{envelope}.
 Since $Q_{\la} | X_1$ is a square  for general $\la$, it is so  for every $\la$.
 
\end{proof}

We add in the next section some considerations about the geometry of smooth  cubic curves  $C_1$ in the plane.

\section{Rank $2$  conics tangent to  a smooth plane cubic $C_1$}

If $C_1$ is a smooth cubic, there are $9$ flex points $P \in C_1$, the points such that   $ 3 P \equiv H$, and we may fix one of them as the origin, let us denote  it   $O$.

Assume for the rest of  this section  that  $Q$ is a conic such that $ Q \cdot C_1 = 2 \de$, where $\de$ is an effective divisor of degree $3$.

There are three possible cases:

1) If $Q$ has rank $1$, then $Q = 2 H_0$, where $H_0 \in |H|$, and $\de = div(H_0)$.

3) If $Q$ has rank $3$, then $ \de \equiv H + \eta$, where $ \eta \in Pic^0(C_1) $  has torsion order exactly $2$.

2)  If $Q$ has instead rank $2$, then again $ \de \equiv H + \eta$, but we can be more precise.

Write $Q = H_1 + H_2$, and letting $\de_j$ be the divisor cut by $H_j$ on the curve $C_1$,
we see that necessarily $ P_Q : = H_1 \cap H_2 \in C_1$. Because otherwise the supports of the two divisors 
$\de_1, \de_2$ would be  disjoint and $\de_1 + \de_2 = 2 \de$ is not possible.

Hence   the line $H_j$ cuts on $C_1$ a divisor  $ \de_j = P_Q  + 2 P_Q(j)$, so $H_j$  is tangent to $C_1$ at the point $P_Q(j)$ (clearly $P_Q(1) \neq P_Q(2)$). 

Then  the divisor $\de$ is linearly equivalent to $ P_Q + P_Q(1) + P_Q(2).$

In the situation we are interested in,  there are conics $Q_{\la}$ of rank $2$ in the pencil, and for this reason  we would like 
to describe in  more  detail the latter case 2).

\begin{theorem}
Let $C_1$ be  a smooth cubic curve, embedded by the linear system $H$ as $C_1 \subset |H|^{\vee} \cong \PP^2$.
Then the variety $\sC \subset |2H|$ of rank $2$ conics tangent to $C_1$ consists of three disjoint smooth
curves $\sC_{\eta}$ of degree $6$ and genus $1$, each corresponding to a nontrivial divisor class $ \eta \in Pic^0(C_1) $  with torsion order exactly $2$.

Moreover $\sC_{\eta}$ is isomorphic to a cubic curve in $|H+\eta|$, isomorphic to $C_1 / \eta$, and is embedded in $|2H|$ via the Veronese embedding
of $|H + \eta|$  inside $v_2(|H + \eta|) \subset |2H|$.
\end{theorem}

\begin{proof}
Here $H_j $ cuts the divisor $\de_j =  P_Q + 2 P_Q(j) \equiv H$ and $ 2 \de = \de_1 + \de_2$, hence $$ \eta : = P_Q(1) - P_Q(2) \Rightarrow 2 \eta \equiv 0, \ \ P_Q \equiv H -  2 P_Q(1) \equiv H -  2 P_Q(2).$$

Hence, once $\eta \in Pic(C_1)[2]\setminus \{0\}$ is fixed,  such a conic $Q$ of rank $2$ is determined by $P' : = P_Q(1)$, or, equivalently,  by $P'': = P_Q(2)$.

Hence we get a map (here $v_2$ is the second Veronese map) $$ d :  C_1 / \eta \ra  |H + \eta| \ra v_2 (|H + \eta|) \subset S^2(|H|) \subset | 2 H|,$$
defined as follows.

The map $d$ is induced  by the map $d'  : C_1 \ra |H + \eta| $ obtained in this way:
for  each point $P' \in C_1$,   take $P $ to be the third  intersection 
point of the tangent  $T_{P'}$ (to $C_1$ at $P'$) with $C_1$,
and define, for $P' \in C_1$,  $ P'' :\equiv P' + \eta$:   then 
$$ P' \mapsto d'(P') := \de(P') : = P + P' + P'' \in |H + \eta|.$$
And $$  q : C_1 \ra S^2(|H|) \ {\rm is \ such \ that  } \ q(P') = q(P'') = T_{P'}+ T_{P''}.$$

It is interesting to observe that  the map $d $ is injective. Indeed, for each point $P \in C_1$, projection from $P$ to $\PP^1$
(via $ |H - P|$) is a double covering, branched in $4$ points, which are the further points of intersection
of the tangents from $P$ to $C_1$.

If $P$ is not a flex point, then the $4$ tangency points $P(1), P(2), P(3), P(4)$ are different from $P$,
and the pairwise differences determine three nontrivial elements $\eta$ of 2-torsion order.

Without loss of generality, since $\eta_j :  \equiv P(1)- P(j)$ and $ \eta_3 \equiv \eta_1 + \eta_2$, 
we have $P(j)- P(h) \equiv \eta_j + \eta_h$.

The same argument holds if $P$ is a flex: in this case  $P(1) = P$.

Hence for each point $P$ we get, for fixed $\eta$, exactly  two such conics $Q$ of rank $2$ with $P = P_Q$, according to the fact that $C_1 / \eta \ra C_1/ C_1[2]$
has degree $2$. 

 We shall soon see  that  injectivity holds  also schematically, and  $ d (C_1 / \eta) $ is a smooth cubic $C'_{\eta} \subset |H+\eta|$;
and, as a consequence, given a general line  $\Lambda \subset |H+\eta|$, $\Lambda \cap  C'_{\eta}$ consists of three points.

Observe in fact  that the Veronese surface $v_2 (|H + \eta|) $ consists of the set  
$$ v_2 (|H + \eta|)  = \{ 2 \de | \de \in |H + \eta| \} \subset |2 H| .$$

Inside the projective space of conics $|2 H|$ we have the discriminant hypersurface $\Sigma$ consisting of
the reducible conics: $\Sigma$ contains the other Veronese surface $V: = v_2 (|H |) $, which is the variety
of double lines, and $\Sigma \setminus V$ admits an  unramified double covering which associates to each
rank 2 conic the two ordered pairs of lines which make the conic (and indeed the intersection of $\Sigma$ with a 
general plane $\pi$  yields a plane cubic given together with a 2-torsion divisor).

We can view the cubic determinantal hypersurface $\Sigma$ as the symmetric 2-fold product $S^2(|H|)$,
there is in fact a map $S^2 : |H|  \times |H| \ra S^2(|H|)$ which ramifies exactly on the diagonal (which maps to 
the Veronese surface $V$). Observe that the map which is the composition of $S^2$ with the inclusion in $|2H|$ 
is a bilinear map of total degree $6 = (1,1)^4$,
which however factors through the degree 2 covering $S^2$, and in fact $\Sigma$ has degree 3.

Consider now the set theoretical intersection of the Veronese surface $v_2 (|H + \eta|) $ with $\Sigma$:
it yields  the rank 2 conics such that their divisors are of the form $2 \de$, for $\de \in |H + \eta|$.
By our previous analysis, this is the image $q( C_1)$, that is, 
the Veronese embedding $\sC_{\eta}$ of $C'_{\eta}$.

 The  set theoretic  intersection $v_2 (|H + \eta|) \cap \Sigma$ is  
a curve of  degree $6$, corresponding then to a cubic curve  inside $|H + \eta|$

Another  way to see this is to go back to  the Gauss map $\ga$ of $C_1$ to its dual sextic curve $C_1^{\vee}$, 
associating to $P$ the line $T_P $, viewed as a point $\ga (P)$  in $ |H| $.

The Gauss map  is given by a subsystem of $|2H|$,
because it  is given by the partial derivatives of the equation  of $C_1$.

Hence the map $q : C_1 \ra S^2(|H|)$ is obtained from the composition of the map $i $ associating to $P'$ the pair $(P', P'') \in C_1 \times C_1$,
 with the map $ \ga \times \ga :  C_1 \times C_1 \ra |H|  \times |H|$ followed by symmetrization  in order to obtain 
$$ C_1 / \eta \subset S^2 (C_1) \ra S^2(|H|) \subset |2H|.$$
Hence $q$ is given by a linear series of degree $12$, hence the degree of $d'$ is $6$ and the degree of $d$ equals $3$.
We conclude that $C_1 / \eta \ra C'_{\eta}$ is an injective map onto a cubic curve, hence $C'_{\eta}$ is smooth
and $d$ yields  an isomorphism with $C'_{\eta}$.

We conclude the proof observing that clearly the $4$ Veronese surfaces $v_2 (|H + \eta|)$ are disjoint,
for different $\eta$'s  of $2$-torsion.

\end{proof}

\begin{corollary}
Consider now a line $\Lambda \subset |H+\eta|$, and let $\Lambda \cap C'_{\eta}$ be a  divisor consisting of three points 
corresponding to rank $2$ conics $Q_1, Q_2, Q_3$. We then have three points $P_1, P_2, P_3 \in C_1$,
which are the kernels of the corresponding symmetric matrices ($P_j$ is, as before, the  double point of $C_j$).

Then the three points $P_1, P_2, P_3 \in C_1$ are collinear.
\end{corollary} 

\begin{proof}

The three points $P_1, P_2, P_3 \in C_1$ are collinear if and only if 
$$H \equiv  \sum_j P_j \Leftrightarrow 2H \equiv  \sum_j 2 P' _j  \ (\equiv \sum_j 2 P'' _j ).$$

The system $\Lambda$ maps to a conic $\sQ \subset 
 v_2 (|H + \eta|) \subset |2H|$, and its intersection with the discriminant hypersurface $\Sigma$
yields  the three points with multiplicity $2$. Moreover,   $\Lambda \subset |H + \eta|$ maps under the Veronese map to  the intersection of $v_2 (|H + \eta|)$ with a  special 
hyperplane, cutting $\Lambda$ with multiplicity 2.  Such a hyperplane 
 pulls back to a bilinear form on $\PP^2 \times \PP^2$ (which maps to $\Sigma$).

This shows that the three intersection points, counted with multiplicity 2,  correspond to the intersection of $\iota (C_1)$ with a
divisor on $C_1 \times C_1$ of the form $\pi_1^* (2H) + \pi_2^*(2H)$;   hence it follows that 
$$4H  \equiv  \sum_j 2 (P' _j + P''_j) \equiv  \sum_j 2 (H -  P_j).$$

This  implies only  the weaker assertion  that $\sum_j  2 P_j \equiv 2H$. 

Our stronger assertion, that $\sum_j   P_j \equiv H$, 
  follows from  
    Lemma \ref{line-envelope} since 
     the conic $\sQ \subset |2H|$ yields a quadratic pencil of conics $Q_{\la}$ tangent to $C_1$, hence the three points $P_1, P_2,P_3$ belong to a line.
\end{proof}

\bigskip

\section{Explicit equations    of 
special quadratic  pencils of conics starting  from the Weierstrass normal form of a cubic} 

In this section we consider a cubic $C_1$, with  equation  in Weierstrass canonical form
\begin{equation}\label{weierstrass} C_1 : = \{ x | x_2^2 x_3 -  x_1 ( x_1 - x_3) (x_1 - \mu x_3) =0\}
\end{equation}

and we want to explicitly determine the 1-dimensional systems of everywhere tangent conics $Q_\la$, 
and also the equation of the line $X_1$ everywhere tangent to the pencil.

$ C_1$ is smooth if $\mu \neq 0,1$, while  for $\mu=0, 1$ we get a nodal cubic.

The 2-torsion points different from the flex at infinity ($(0,1,0)$ with flex tangent $\{x_3=0\}$) are the  points on the line
$x_2=0$, namely $ (0,0,1), (1,0,1), ( \mu, 0, 1)$. 

We observe first  that translation by a 2-torsion divisor class $\eta$ carries the system $|H|$ to the system $|H + \eta|$.

Then the embedding corresponding to the system $|H + \eta|$ is the composition  $i \circ \tau_\eta$, where 
$i :  C_1 \ra |H|^{\vee}$ is the inclusion we started with, and $\tau_\eta$ denotes translation by $\eta$.

We may assume, without loss of generality,  that $\eta$ corresponds to the point  $(0,0,1),$
which, by slight abuse of notation, will be called $\eta$.

\begin{theorem}\label{weierstrass}

Given the plane cubic in Weierstrass normal form
$$C_1 : = \{ x | x_2^2 x_3 -  x_1 ( x_1 - x_3) (x_1 - \mu x_3) =0\},$$

with  $\mu\notin\{0,1\}$, origin at the point at infinity $(0,1,0)$, and with 2-torsion divisor associated to the
point $\eta= (0,0,1),$
consider the embedding of $C_1 \ra C_1'$ given by the linear system $|H + \eta|$.
This is such that 

$$P + \eta =   ( \mu x_1^2 , - \mu x_2 x_1, - \mu x_1 x_3 + x_2^2  +   ( 1 + \mu)x_1 ^2 ) = : (z_1,z_2,z_3).
$$

Then to the    pencil $\Lambda$ on $C_1' \subset \PP^2$  cut by the linear pencil of lines with base point $(1, w_2, w_3) \in \PP^2  \setminus C_1'$
corresponds a special quadratic pencil of conics,  with cubic $C_1$ and residual line 

\[ X_1(x) = \frac{\text{Disc}(Q_\la)(x)}{C_1(x)} = 4\mu^2 (-w_2^2 - 2 w_3 \mu + \mu + 1) x_1  -  8w_2 \mu^2 x_2 + 4 \mu^3 (w_3^2 \mu - 1) x_3 .\]

The equations of the quadratic pencil are given in equation \eqref{Q}.
\end{theorem}

\begin{proof}
If $P = (x_1, x_2, x_3)$, then $ P + \eta$ is the third intersection point with $C_1$ of the line through 
$-P = (x_1, -x_2, x_3)$ with  $(0,0,1),$ which equals the line 
$$\{ (t_0 x_1, - t_0  x_2, t_0 x_3 + t_1)\} .$$

Solving the equation $$  t_0^2 x_2^2 (t_0 x_3+t_1)  -  t_0 x_1 ( t_0 x_1 - t_0 x_3-t_1 ) (t_0 (x_1 - \mu  x_3) - \mu t_1) =0$$
 we get the obvious roots  $t_0=0$,  $t_1=0$, hence we only look at the terms divisible by $t_0 t_1$;
 and dividing by $t_0 t_1$ we get 
$$t_0 [ x_2^2  + x_1 ^2 ( 1 + \mu) - 2 \mu x_1 x_3 ] =  t_1 \mu x_1.$$

This is solved by $ t_0 = \mu x_1, t_1 =  [ x_2^2  + x_1 ^2 ( 1 + \mu) - 2 \mu x_1 x_3 ] $.

Hence, as asserted,
 \begin{equation}\label{translation}P + \eta =   ( \mu x_1^2 , - \mu x_2 x_1, - \mu x_1 x_3 + x_2^2  +   ( 1 + \mu)x_1 ^2 ) = : (z_1,z_2,z_3).
   \end{equation}

We have set for convenience $P + \eta = : (z_1(x), z_2(x), z_3(x))$.

The reader will observe that the given linear system of quadratic polynomials $(z_1(x), z_2(x), z_3(x))$ has a base point for $x_1=x_2=0$, 
that is, at the point $\eta$,
and with  multiplicity $3$  (since at this point, setting $x_3=1$,  the subscheme $x_1^2= x_1x_2= - \mu x_1  + x_2^2$ is a curvilinear subscheme on the curve $\{  \mu x_1  = x_2^2 \}$,  of  length 3 ).

Then a  linear pencil $\Lambda \subset \PP(|H + \eta|)$ corresponds to a pencil $ \la_1 F' (z) + \la_2 F''(z)=0$,
where $F' , F''$ are linear forms.

And the conic $Q_\la$ is tangent to  $C_1$ in three points (with multiplicity 2),
 which are the residual  intersection of $C_1$ with 
 $$ \{ 0 = (\la_1 F' (z) + \la_2 F''(z))^2\} \Leftrightarrow $$
 $$ \{ 0 = (\la_1^2 (F')^2 + 2 \la_1 \la_2 F' F'' + \la_2^2 (F'')^2) ( \mu x_1^2 , - \mu x_2 x_1, - \mu x_1 x_3 + x_2^2  +  ( 1 + \mu) x_1 ^2  )\}.$$

By residual intersection we mean that we remove the divisor consisting of the point $\eta$ counted with multiplicity $6$ (which appears in the divisor of zeros on $C_1$ of the above  
quartic equation).

\bigskip
\subsection{Determination of the conics $Q_\la$ }

The conic  $Q_\la$ can be found  as follows: 

\medskip

1) we determine a conic $q_\eta$ such that it intersects $C_1$ only at the point $\eta$ (hence with multiplicity $6$):
\begin{equation}
 q_\eta = - \mu  x_1x_3   + (1 + \mu)  x_1^2  +  x_2^2
\end{equation} 

2) we look for a linear form $l(x) = l_1 x_1 + l_2 x_2 + l_3 x_3$ such that,   setting $F_\la^2 : = (\la_1^2 (F')^2 + 2 \la_1 \la_2 F' F'' + \la_2^2 (F'')^2)$,
then  $ (F_\la^2  - l(x) C_1)$
is divisible by $q_\eta$.

3) Then we take $Q_\la$ as the result of division of $ (F_\la^2  - l(x) C_1)$
 by $q_\eta$.

\medskip

For step 1), take local coordinates in the plane $(x_1, x_2)$ at the point $\eta$ (setting $x_3=1$):
then the local equation of $C_1$ is 
$$ x_2^2 - x_1 (\mu - (1+\mu) x_1 + x_1^2) =0.$$

Hence $x_2$ is a local parameter for $C_1$ at $\eta$, and 
$$x_1 = (1/\mu )[ x_2^2 + (1+\mu ) x_1^2 - x_1^3 ]= (1/\mu ) [ x_2^2 + h.o.t ]$$ vanishes of order $2$.

We calculate the Taylor development of $x_1$ up to order $5$, this shall be sufficient for our purposes:

$$x_1 = (1/\mu )[ x_2^2 + (1+\mu ) x_1^2 + o(5) ]= \frac{1}{\mu }  x_2^2 + \frac{ 1+\mu }{\mu^3} x_2^4  + o(5) .$$ 

Then, writing  $q_\eta$ in affine coordinates (since it must vanish at the origin) as
$$ a_1 x_1 + a_2 x_2 + b_1 x_1^2 + b_2 x_2^2 + c x_1 x_2,$$ 
we determine the coefficients by the fact that the Taylor development of order up to $5$ is identically zero:
$$ (a_1+ c  x_2)  [\frac{1}{\mu }  x_2^2 + \frac{ 1+\mu }{\mu^3} x_2^4]  + a_2 x_2 + b_2 x_2^2  + b_1 \frac{1}{\mu^2 }  x_2^4  = o(5) .$$

This easily yields $a_2=0$, $b_2 = - \frac{1}{\mu }  a_1$, $c=0$.

To simplify the expressions, multiply everything by $\mu^3$, getting 
$$ a_1  (1+\mu ) x_2^4     + b_1 \mu  x_2^4  = o(5)  \Leftrightarrow a_1 = - \mu , b_1 = 1 + \mu .$$

In the end, 
$$ q_\eta = - \mu  x_1  + (1 + \mu)  x_1^2 +  x_2^2 .$$ 

Step 2) The process of division by $ q_\eta$ of a polynomial $F(x)$ is quite simple: it is just the Euclidean algorithm in the
ring $\CC[x_1, x_3 ] [x_2]$, applied to the monic polynomial 
$$ q_\eta = (- \mu  x_1x_3   + (1 + \mu)  x_1^2 ) +  x_2^2 = : A(x_1, x_3 ) + x_2^2  .$$

The main concern here is to find the linear form $l(x)$. The pencil $\Lambda$ is determined by its base point $(w_1, w_2, w_3)$, and we consider  the inhomogeneous case where 
$w_1=1$.

{\bf Case where $w_1 \neq 0$.}

Then $ F_\la = \la_1 (z_3 - w_3 z_1) + \la_2 (z_2 - w_2 z_1) = \la_1 z_3 + \la_2 z_2 - ( \la_1 w_3 + \la_2 w_2)  z_1 $ and
seeing  the $z_i$'s as functions of $(x)$ we get  
$$  F_\la =  \la_1   (- \mu x_1 x_3 + x_2^2  +  x_1 ^2 ( 1 + \mu) ) - \la_2  \mu x_2 x_1 - ( \la_1 w_3 + \la_2 w_2)  \mu x_1^2
= $$
$$ = \la_1 q_\eta - x_1 (  \la_2  \mu x_2 + ( \la_1 w_3 + \la_2 w_2)  \mu x_1) = : \la_1 q_\eta - x_1 \mu B (x_1, x_2).
$$
Hence, 
$$ F_\la ^2 =  [\la_1 q_\eta - x_1 \mu B (x_1, x_2)]^2 = q_\eta (\la_1^2  q_\eta - 2 \la_1 x_1 \mu B) + x_1^2 \mu^2 B^2 .$$

The divisor cut on $C_1$ by $(x_1 B)$ equals 3 times the point $\eta$, plus  the flexpoint $O$ at infinity,
plus  the two points $P_1, P_2$ which, together with the point $\eta$,
are the intersection of $C_1$ with the line  $B $,
where we  have set  
\begin{equation}\label{B,c}
c : = c (w, \la) : =  ( \la_1 w_3 + \la_2 w_2) , \ \ B  = :  \la_2   x_2 + c   x_1
\end{equation}

In fact, the points $P_1, P_2,$  satisfy, in affine coordinates
  $$  (c  x_1)^2  =  (\la_2    x_2)^2 = 
  (\la_2   )^2 x_1 (x_1 -1 ) (x_1 - \mu ),$$
hence correspond to the roots of the quadratic polynomial
$$ f_\la : = - c^2     x_1 +  \la_2  ^2 (x_1 -1 ) (x_1 - \mu ),$$
on the line  $B$. 

We  determine  now a conic $N_\la$ such that the divisor of $N_\la q_\eta$ on $C_1$ equals 
$$ div((x_1 B)^2) = 6 \eta + 2 O + 2 P_1 + 2 P_2.$$

 Hence $ div(N_\la) = 2 O + 2 P_1 + 2 P_2$.
 
  $N_\la$, which passes through $O$, is   tangent at $O$, hence also to the line at infinity, hence  it has the form
  of a multiple of 
  $$N'_\la = x_3 g_\la (x) + x_1^2 = x_3 (g_1 x_1 + g_2 x_2 + g_3 x_3) + x_1^2.$$
  The intersection points with the line $B$ are given by the roots 
  of 
  $ x_3 (g_1 x_1 - \frac{c}{\la_2}  g_2 x_1 + g_3 x_3) + x_1^2 $,
  in affine  coordinates  $$   \frac{g_1 \la_2 - c g_2 }{\la_2}  x_1 + g_3  + x_1^2 $$
  which must equal up to a function of $\la_1, \la_2$ 
  $$- c^2  x_1 + \la_2  ^2 (x_1 -1 ) (x_1 - \mu )= \la_2^2 ( x_1^2 - (1 + \mu + \frac{c^2}{\la_2^2} )x_1  + \mu) .$$
  
  Hence  we set $ g_3 = \mu  ,   \frac{g_1 \la_2 - c g_2 }{\la_2}  = - (1 + \mu + \frac{c^2}{\la_2^2} ).$
  
  The last equation can be rewritten as
 $$ \la_2^2  g_1  - \la_2 c g_2  = - \la_2^2(1 + \mu)  - c^2  .$$
 
 To preserve the homogeneity in $\la_1, \la_2$ we set $N_\la : = \la_2^2 N'_\la$.

 We pass now to the full  equations for the $g_\la$'s, ensuring  that the difference of a  multiple of $ N_\la q_\eta $ subtracted by   $ \mu^2 x_1^2 B^2$
 is divisible by $C_1$.  That is:
 $$ d N'_\la q_\eta - \mu^2 x_1^2 B^2 = l(x) C_1(x) $$


 Then $$Q_\la = \la_1^2 q_\eta - 2 \la_1 \mu x_1 B(x_1, x_2) + d (x_3 g_\la + x_1^2)=$$
 $$ = \la_1^2 (- \mu  x_1x_3   + (1 + \mu)  x_1^2 +  x_2^2 )  - 2 \la_1 \mu x_1 B(x_1, x_2) + d (x_3 g_\la + x_1^2) .$$

\subsection{Finding the linear form $l(x)$}
To determine $l(x) =  l_1 x_1 + l_2 x_2 + l_3 x_3$ we solve the equations in the variables $l_1, l_2, l_3$ looking at the monomials of degree 4 ordered in lexicographic order:

\begin{enumerate}
\item
$- l_1 = d (1+ \mu) - \mu^2 c^2$
\item
$l_2 =  2 \la_2 \mu^2 c$
\item
$l_1 (1 + \mu) - l_3 = d g_1 (1+\mu)  - \mu d$
\item
$d - \la_2^2 \mu^2 =0$
\item
$ l_3 (1+ \mu) - l_1 \mu =  d \mu (1 + \mu - g_1)$
\item
$ l_2  (1+ \mu) = (1+ \mu) d g_2$
\item
$ l_1 = g_1 d$
\item
$l_2 \mu = \mu g_2 d$
\item
$l_3 \mu = \mu^2 d$
\item
$l_2 = d g_2$
\item
$l_3 = \mu d$.  
\end{enumerate} 

Equation (4) tells that $d =  \la_2^2 \mu^2$, (7), (10) and (11) determine $l_1, l_2, l_3$,
and (9), (8), (6) follow right away. (2) says that $\la_2 g_2 = 2 c$, which is new information.

In order to verify the other equations it may be useful to recall:
$$ g_3 = \mu, \ \ \la_2 g_2 = 2c , \ \ \la_2^2  (g_1+1 + \mu) =   \la_2 c g_2    - c^2  = c^2  \Leftrightarrow g_1 = - (1 + \mu) + \frac{c^2}{\la_2^2} $$
From this follow immediately (1) and (2). While (3) and (5) are identically verified.

\subsection{The conics $Q_\la$}
$$ Q_\la =  \la_1^2 (- \mu  x_1x_3   + (1 + \mu)  x_1^2  +  x_2^2 )  - 2 \la_1 \mu x_1 ( cx_1 + \la_2  x_2) + \la_2^2  \mu^2  (x_1^2 + x_3 (\mu x_3 + \frac{2 c x_2}{\lambda_2} + g_1 x_1) ) =
$$
$${\rm with \ }  c : = ( \la_1 w_3 + \la_2 w_2),$$
\begin{equation}\label{Q}
  x_1^2 [ \la_1^2 (1 + \mu) - 2 \la_1 c \mu + \la_2^2 \mu^2 ] - x_1 x_2 [2 \la_1 \la_2 \mu ] + x_1 x_3 [g_1 \la_2^2  \mu^2 - \mu \la_1 ^2 ] + x_2^2 [\la_1^2  ] + x_2 x_3 [ 2 c \la_2 \mu^2  ] + x_3^2 [\la_2^2 \mu^3 ] \end{equation}

\bigskip

It remains to be verified whether  this 1-parameter system of conics consists of conics tangent to a fixed line,
and that there are exactly 3 degenerate conics in the system, each counted with multiplicity 2
 (that is, each yielding a double root of the discriminant equation).

For $\la_1=0$ we get a  nondegenerate conic  if $w_2\neq0$, because if moreover we assume $\la_2=1$, then the determinant equals 
$ - w_2^2 \mu^6 $; similarly for $\la_1=1$ and $\la_2=0$ the determinant equals 
$  -\frac{\mu^2}{4}(\mu w_3^2-1)^2 $.

Hence we set for convenience $\la_1 =1$ and $\la_2 = \la$, hence then $c = w_3 + \la w_2$ :
$$ Q_\la =  x_1^2 [ (1 + \mu) - 2  c \mu + \la^2  \mu^2 ] - x_1 x_2 [2 \la \mu ] + x_1 x_3 [g_1 \la^2 \mu^2 - \mu ] + x_2^2 [1  ] + x_2 x_3 [ 2 c \la \mu^2  ] + x_3^2 [\la^2 \mu^3 ] 
$$

The determinant of $Q_\lambda$ is given by
\begin{align*} det&(Q_\lambda) = det
  \begin{bmatrix}
    (1 + \mu) - 2c\mu + \lambda^2 \mu^2 & -\lambda \mu & \frac{1}{2} (g_1 \lambda^2 \mu^2 - \mu) \\
    -\lambda \mu & 1 & c \lambda \mu^2 \\
    \frac{1}{2} (g_1 \lambda^2 \mu^2 - \mu) & c \lambda \mu^2 & \lambda^2 \mu^3
  \end{bmatrix} = \\ 
  =& \  \frac{\mu^2}{4} \cdot \Big( ( -4w_2^2\mu^4 )\lambda^6 +  ( 4w_2^3\mu^3 - 8w_2w_3\mu^4 + 4w_2\mu^4 + 4w_2\mu^3 ) \lambda^5 \\
  &+ ( -w_2^4\mu^2 + 12w_2^2w_3\mu^3 - 2w_2^2\mu^3 - 2w_2^2\mu^2 - 4w_3^2\mu^4 + 4w_3\mu^4 + 4w_3\mu^3 - \mu^4 - 2\mu^3 - \mu^2 ) \lambda^4 \\
  &+  ( -4w_2^3w_3\mu^2 + 12w_2w_3^2\mu^3 - 4w_2w_3\mu^3 - 4w_2w_3\mu^2 - 4w_2\mu^2 ) \lambda^3 \\
  &+ ( -6w_2^2w_3^2\mu^2 + 2w_2^2\mu + 4w_3^3\mu^3 - 2w_3^2\mu^3 - 2w_3^2\mu^2 - 4w_3\mu^2 + 2\mu^2 + 2\mu ) \lambda^2 \\
  &+  ( -4w_2w_3^3\mu^2 + 4w_2w_3\mu ) \lambda + ( -w_3^4\mu^2 + 2w_3^2\mu - 1 ) \Big). 
 \end{align*}

The determinant of the matrix is the square of the polynomial 
$$\delta(\lambda) =  \frac{1}{2} i \mu \Big(  ( -2w_2\mu^2 ) \lambda^3 + ( w_2^2\mu - 2w_3\mu^2 + \mu^2 + \mu )\lambda^2 +  ( 2w_2w_3\mu )  \lambda + ( w_3^2\mu - 1 ) \Big) .$$  

 Furthermore, $Q_\la$ can be rewritten in the form:
 \begin{align*} Q_\lambda =&  \ ( \mu^2\, x_1^2 + w_2^2 \mu^2\, x_1 x_3 - \mu^3\, x_1 x_3 - \mu^2\, x_1 x_3  {}+ 2w_2 \mu^2\, x_2 x_3 + \mu^3\, x_3^2 ) \lambda^2 \\
&+ (  -2w_2 \mu\, x_1^2 - 2\mu\, x_1 x_2  {}+ 2w_2 w_3 \mu^2\, x_1 x_3 + 2w_3 \mu^2\, x_2 x_3 ) \lambda  \\
&+ ( -2w_3 \mu + \mu + 1 ) x_1^2 + ( w_3^2 \mu^2 - \mu ) x_1 x_3 + x_2^2.
\end{align*}

Consequently, computing the discriminant with respect to $\la$ yields:
\[ X_1(x) = \frac{\text{Disc}(Q_\la)(x)}{C_1(x)} = \mu^2 [4  (-w_2^2 - 2 w_3 \mu + \mu + 1) x_1  - 8w_2 x_2 + 4 \mu (w_3^2 \mu - 1) x_3.] \]

 The discriminant is not identically zero for $\mu \neq 0$, and for $\mu=0$ we get a limit line.

Since $(1,w_2,w_3)\notin C_1'$, the conics  $Q_\la$ do not all pass through  the same   point lying  on $C_1'$, hence Lemma \ref{line-envelope} applies.

\end{proof}
 
\bigskip

{\bf Acknowledgements: } We thank Meng Chen for asking the question which motivated the present work,
and we are very grateful to the referees for corrections and useful suggestions which have hopefully made the presentation of our arguments simpler and more understandable.

\end{document}